\documentclass[journal]{IEEEtran}
\usepackage{cite}
\usepackage{amsmath,amssymb,amsfonts,amsthm}
\usepackage{algorithm}
\usepackage{algorithmic}
\usepackage[algo2e,linesnumbered,ruled]{algorithm2e}
\usepackage{graphicx}
\usepackage{textcomp}
\def\BibTeX{{\rm B\kern-.05em{\sc i\kern-.025em b}\kern-.08em
    T\kern-.1667em\lower.7ex\hbox{E}\kern-.125emX}}
\usepackage{float}
\usepackage{color}

\usepackage[utf8]{inputenc}
\usepackage{mathtools}  
\usepackage{nccmath}

\usepackage{dsfont}
\usepackage{caption}
\usepackage{subcaption}
\theoremstyle{plain}
\newtheorem{theorem}{Theorem}
\newtheorem{definition}{Definition}
\newtheorem{assumption}{Assumption}

\newtheorem{remark}[theorem]{Remark}

\newtheorem{property}[theorem]{Property}

\newcommand{\mR}{{\mathbb R}}

\newcommand{\mP}{{\mathbb P}}

\newcommand{\mU}{{\mathbb U}}

\newcommand{\bF}{{\mathbf F}}

\newcommand{\bP}{{\mathbf P}}

\newcommand{\bM}{{\mathbf M}}

\newcommand{\bX}{{\mathbf X}}

\newcommand{\bx}{{\mathbf x}}

\newcommand{\bPsi}{{\boldsymbol \Psi}}

\newcommand{\bY}{{\mathbf Y}}
\newcommand{\bg}{{\mathbf g}}
\newcommand{\bA}{{\mathbf A}}

\newcommand{\bs}{{\mathbf s}}

\newcommand{\bc}{{\mathbf c}}
\newcommand{\bv}{{\mathbf v}}
\newcommand{\bw}{{\mathbf w}}

\newtheorem*{proof}{Proof}
\begin{document}

\title{A Convex Approach to Data-driven Optimal Control via Perron-Frobenius and Koopman Operators}
\author{Bowen Huang, Umesh Vaidya 
\thanks{
The research work was supported from NSF CPS award 1932458 and NSF award 2031573. U. Vaidya and B. Huang are with the Department of Mechanical Engineering, Clemson University, Clemson, SC 29631 USA (e-mail: uvaidya@clemson.edu).}}
 
\maketitle


\begin{abstract}

The paper is about the data-driven computation of optimal control for a class of control affine deterministic nonlinear systems. We assume that the control dynamical system model is not available, and the only information about the system dynamics is available in the form of time-series data. We provide a convex formulation for the optimal control problem of the nonlinear system. The convex formulation relies on the duality result in the dynamical system's stability theory involving density function and Perron-Frobenius operator. We formulate the optimal control problem as an infinite-dimensional convex optimization program. The finite-dimensional approximation of the optimization problem relies on the recent advances made in the Koopman operator's data-driven computation, which is dual to the Perron-Frobenius operator. Simulation results are presented to demonstrate the application of the developed framework.
\end{abstract}


\section{Introduction}

The development of data-driven approaches for solving optimal control problems (OCP) for the dynamical system has attracted tremendous interest from various research communities. 
The solution to OCP involves solving an infinite-dimensional nonlinear partial differential equation, namely Hamilton Jacobi Bellman (HJB) equation. 
The nonlinear and infinite-dimensional nature of the HJB equation makes the OCP challenging. Alternate formulations of OCP have been sought, which are more amicable to data-driven computation.

Progress is made for a class of linearly solvable OCP using alternate Kullback-Leibler (KL) based formulation of OCP for a stochastic dynamical system, and path integral-based numerical scheme \cite{kappen2007introduction,todorov2009efficient,theodorou2010generalized}. 
The convex approach to the OCP we propose for data-driven control draws a parallel with the following literature on this topic. 
A convex formulation to OCP as an infinite-dimensional linear program is proposed in \cite{hernandez,gaitsgory2009linear}. Duality relationship is used for convex formulation of OCP in \cite{vinter1993convex}. In \cite{Rantzer01}, the density function is introduced as dual to Lyapunov function for verifying weaker notion of almost everywhere (a.e.) stability. The density function is used in the convex formulation of stabilization and optimal control \cite{Prajna04,rantzer2003duality}. Similarly, convex dual formulation involving occupation measures via moment-based approach for control problem is studied in \cite{henrion2013convex,korda2016moment,korda2017convergence,lasserre2008nonlinear}.  One of our work's main distinguishing features is that we view the duality in the stability theory and control using tools from the linear operator theory. This viewpoint allows us to present novel convex formulation to OCP with variables having physical interpretation and is dynamical system motivated.
Furthermore, given the infinite time horizon nature of our OCP formulation, stability is central in our formulation. We introduce a new notion of a.e. uniform stability, stronger than a.e. stability \cite{Rantzer01,Vaidya_TAC}, and show that optimal control satisfies this property \cite{rajaram2010stability}.  Another main advantage of the linear operator viewpoint is that it allows us to use recent advances in the data-driven approximation of linear operators to provide a data-driven solution to the OCP. The operator theoretic viewpoint also enables us to extend the convex formulation to optimal control from deterministic to stochastic setting in a straightforward manner \cite{Vaidya_stochastic2020}.

The linear P-F and Koopman operators are used to lift nonlinear dynamics from state space to linear, albeit infinite-dimensional, dynamics in the space of functions. 
There is a growing body of literature on using these operators, especially the Koopman operator, for dynamical analysis and control. We refer the readers to the following references for more details on this topic \cite{kaiser2017data,huang2018feedback,arbabi2018data,ma2019optimal,korda2020optimal,mezic_koopman_stability,huang2020data}. Our approach differs from the existing methods involving the use of the Koopman operator for data-driven control. In particular, we perform bilinear lifting of the control system using the P-F operator, which is dual to the Koopman operator.
We address the control challenges associated with bilinear lifting of the control system using the convexity property in the dual-density space and the positivity property of the linear operators
\cite{Vaidya_TAC, Vaidya_CLM_journal,raghunathan2013optimal}. While the theoretical formulation of OCP is based on the P-F operator, the numerical framework exploits the advantage of the Koopman operator's data-driven computation.

The main contributions of the paper are as follows. We present a systematic framework based on the linear operator theory for the data-driven optimal control of a class of continuous-time nonlinear systems. 
The lifting involving the P-F operator is instrumental in the convex formulation of the OCP in the dual space of density. 
The stability property is intimately connected to the optimal control, and we show that the optimal control ensures, a.e. uniform stability of closed-loop dynamics. The proposed data-driven computational framework exploits the recent advances in approximating the Koopman operator and the duality between the Koopman and P-F operator.
In particular, the computational framework makes use of the Naturally Structured Dynamic Mode Decomposition (NSDMD)\cite{huang2018data} algorithm for the approximation, preserving positivity and Markov properties of the linear operators. Time-series data from single or multiple trajectories corresponding to a system with zero input and unit step input are used in the training process.

The paper is organized as follows. In Section \ref{section_prelim}, we present some preliminaries on the linear operator theory and NSDMD algorithm for the finite-dimensional approximation of the Koopman and P-F operators. The main results on the formulation of the convex optimization problem for optimal control are presented in Section \ref{section_main}. The computational framework for the finite-dimensional approximation of the OCP is presented in Section \ref{section_data-driven}. Simulation results are presented in Section \ref{secction_simulation} followed by  conclusions in Section \ref{section_conclusion}.



    
    
    


\section{Preliminaries and Notations}\label{section_prelim}

Following notations will be used throughout this paper. $\mR^n$ denotes the $n$ dimensional Euclidean space  and $\mR^n_{\geq 0}$ is the positive orthant. Let, $\bX\subseteq \mR^n$ and $\bY\subseteq \mR^m$ ${\cal L}_1(\bX,\bY), {\cal L}_\infty(\bX,\bY)$, and ${\cal C}^k(\bX,\bY)$ denote the space of all real valued integrable functions, essentially bounded functions, and space of $k$ times continuously  differentiable functions mapping from $\bX$ to $\bY$ respectively. If the space $\bY$ is not specified then it is understood that the underlying space if $\mR$. ${\cal B}(\bX)$ denotes the Borel $\sigma$-algebra on $\bX$ and ${\cal M}(\bX)$ is the vector space of real-valued measure on ${\cal B}(\bX)$. $\bs_t(\bx)$ denotes the solution of dynamical system $\dot \bx=\bF(\bx)$ starting from initial condition $\bx$.

\subsection{Perron-Frobenius and Koopman Operator}
Consider a dynamical system of the form
\begin{eqnarray}
\dot \bx={\bF}(\bx),\;\;\;\bx\in \bX\subseteq \mathbb{R}^n\label{sys}.
\end{eqnarray}
where the vector field is assumed to be $\bF(\bx)\in {\cal C}^1(\bX,\mR^n)$. 
There are two different ways of lifting the finite dimensional nonlinear dynamics from state space to infinite dimension  space of functions namely using Koopman and Perron-Frobenius operators. The definitions of these operators along with the infinitesimal generators of these operators are defined as follows \cite{Lasota}.
\begin{definition}[Koopman Operator]  $\mathbb{U}_t :{\cal L}_\infty(\bX)\to {\cal L}_\infty(\bX)$ for dynamical system~\eqref{sys} is defined as 
\begin{eqnarray}[\mathbb{U}_t \varphi](\bx)=\varphi(\bs_t(\bx)). \label{koopman_operator}
\end{eqnarray}
The infinitesimal generator for the Koopman operator is given by
\begin{eqnarray}
\lim_{t\to 0}\frac{(\mathbb{U}_t-I)\varphi}{t}=\bF(\bx)\cdot \nabla \varphi(\bx)=:{\cal K}_{\bF} \varphi,\;\;t\geq 0. \label{K_generator}
\end{eqnarray}
\end{definition}

\begin{definition} [Perron-Frobenius Operator] $\mathbb{P}_t:{\cal L}_1(\bX)\to {\cal L}_1(\bX)$ for dynamical system~\eqref{sys} is defined as 
\begin{eqnarray}[\mathbb{P}_t \psi](\bx)=\psi(\bs_{-t}(\bx))\left|\frac{\partial \bs_{-t}(\bx) }{\partial \bx}\right|, \label{PF_operator}
\end{eqnarray}
where $\left|\cdot \right|$ stands for the determinant. The infinitesimal generator for the P-F operator is given by 
\begin{eqnarray}
\lim_{t\to 0}\frac{(\mathbb{P}_t-I)\psi}{t}=-\nabla \cdot (\bF(\bx) \psi(\bx))=: {\cal P}_{\bF}\psi,\;\;t\geq 0. \label{PF_generator}
\end{eqnarray}
\end{definition}
These two operators are dual to each other where the duality is expressed as follows.
\begin{eqnarray*}
\int_{\bX}[\mathbb{K}_t \varphi](\bx)\psi(\bx)d\bx=
\int_{\bX}[\mathbb{P}_t \psi](\bx)\varphi(\bx)d\bx.
\end{eqnarray*}

\begin{property}\label{property}
These two operators enjoy positivity and Markov properties which are used in the finite dimension approximation of these operators. 
\begin{enumerate}
    \item Positivity: The P-F and Koopman operators are positive operators i.e., for any $0\leq\varphi(\bx)\in {\cal L}_\infty(\bX)$ and $0\leq \psi(\bx)\in {\cal L}_1(\bX)$, we have
    
    \begin{equation}\label{positive_prop}
        [\mathbb{P}_t\psi](\bx)\geq 0,\;\;\;\;[\mathbb{U}_t\varphi](\bx)\geq 0,\;\;\;\forall t\geq 0.
    \end{equation}
    \item Markov Property:  The P-F operator satisfies Markov property i.e.,
\begin{equation}\label{Markov_prop}
    \int_\bX [\mathbb{P}_t\psi](\bx)d\bx=\int_\bX\psi(\bx)d\bx.
\end{equation}
\end{enumerate}
\end{property}
\begin{definition}[Equivalent Measures]
Two measures $\mu_1\in {\cal M}(\bX)$ and $\mu_2\in{\cal M}(\bX)$ are said to be equivalent i.e., $\mu_1 \approx \mu_2$ provided $\mu_1(B)=0$ if and only if $\mu_2(B)=0$ for every set $B\in {\cal B}(\bX)$. 
\end{definition}
\subsection{Almost everywhere uniform stability and Stabilization}
The formulation for the OCP we present is intimately connected to density function and Lyapunov measure introduced for verifying following stronger notion of almost everywhere (a.e.) uniform stability as introduced in \cite{Vaidya_converse,rajaram2010stability}. We first make following assumption on the system (\ref{sys}).
\begin{assumption}\label{assume_localstability} We assume that $\bx=0$ is locally stable equilibrium point for the system (\ref{sys}) with local domain of attraction denoted by ${\cal N}$. Let $B_\delta$ be the neighborhood of the origin for any given fixed $\delta>0$ such that $0\in B_\delta\subset \cal N$. We let $\bX_1:=\bX\setminus B_\delta$
\end{assumption}

\begin{definition}[Almost everywhere (a.e.) uniform stability] \label{def_aeuniformstable} The equilibrium point is said to be a.e. uniform  stable w.r.t. measure $\mu\in {\cal M}(\bX)$ if for any given $\epsilon$, there exists a time $T(\epsilon)$ such that
\begin{eqnarray}
\int_{T(\epsilon)}^\infty \mu (B_t)dt<\epsilon,\label{eq_aeunifrom}
\end{eqnarray}
where $B_t:=\{\bx \in \bX_1: \bs_t(\bx)\in B\}$ for every set $B\in{\cal B}(\bX_1)$.
\end{definition}
The above stability definition implies that the  system trajectories starting from a.e. initial conditions w.r.t. measure $\mu$ of staying in any set $B\in{\cal B}(\bX_1)$ can be made arbitrary small after sufficiently large time. The main results of this paper prove that the optimal control ensures a.e uniform stability of the equilibrium point. It is proved that a.e. uniform stability is stronger than a.e. stability \cite[Lemma 7 ]{rajaram2010stability}.  The a.e. stability as introduced in \cite{Rantzer01} is defined as follows.

\begin{definition}\label{definition_aestability}
The equilibrium point at $\bx=0$ is said to be a.e. stable w.r.t. measure, $\mu$, if
\[\mu\{\bx\in \bX: \lim_{t\to \infty}\bs_t(\bx)\neq 0\}=0.\]
\end{definition}

In the rest of the paper we are going to make following assumption on measure $\mu$.
\begin{assumption}\label{assumption_mu}
We assume that the measure $\mu\in {\cal M}(\bX)$ is equivalent to Lebesgue with Radon–Nikod\'{y}m derivative $h$ i.e.,  $\frac{d\mu}{d\bx}=h(\bx)>0$ and $h\in {\cal L}_1(\bX,\mR_{> 0})\cap {\cal C}^1(\bX)$.
\end{assumption}
The following theorem is from \cite[Theorem 13]{rajaram2010stability} providing necessary and sufficient condition for a.e. uniform stability. 

\begin{theorem}\label{theorem_necc_suff}
The equilibrium point $\bx=0$ for system (\ref{sys}) satisfying Assumption \ref{assume_localstability} is a.e. uniformly stable w.r.t. measure $\mu$ if and only if there exists a  function $\rho(\bx)\in{\cal C}^1(\bX\setminus \{0\},\mR_{\geq 0})\cap {\cal L}_1(\bX_1)$  and satisfies
\begin{eqnarray}
\nabla\cdot ({\bf F}\rho )=h.\label{steady_pde1}
\end{eqnarray}
\end{theorem}

\subsection{Data-Driven Approximation: Naturally Structured Dynamic Mode Decomposition}\label{section_nsdmd}
Naturally structured dynamic mode decomposition (NSDMD) is a modification of Extended Dynamic Mode Decomposition (EDMD) algorithm \cite{williams2015data}, one of the popular algorithms for Koopman approximation from data. The modifications are introduced to incorporate the natural properties of these operators namely positivity and Markov. For the continuous-time dynamical system (\ref{sys}), consider snapshots of data set obtained as time-series data from single or multiple trajectories
\begin{eqnarray}
{\mathcal X}= [\bx_1,\bx_2,\ldots,\bx_M],\;\;\;\;{\cal Y} = [\mathbf{y}_1,\mathbf{y}_2,\ldots,\mathbf{y}_M] ,\label{data}
\end{eqnarray}
where $\bx_i\in \bX$ and $\mathbf{y}_i\in \bX$. The pair of data sets are assumed to be two consecutive snapshots i.e., $\mathbf{y}_i=\bs_{\Delta t}(\bx_i)$, where $\bs_{\Delta t}$ is solution of (\ref{sys}). Let ${\bPsi}=[\psi_1,\ldots,\psi_N]^\top$ be the choice of basis functions.
The popular Extended Dynamic Mode Decomposition (EDMD) algorithm provides the finite-dimensional approximation of the Koopman operator as the solution of the following least square problem. 

\begin{equation}\label{edmd_op}
\min\limits_{\bf K}\parallel {\bf G}{\bf K}-{\bf A}\parallel_F,
\end{equation}
where,
{\small
\begin{eqnarray}\label{edmd1}
{\bf G}=\frac{1}{M}\sum_{m=1}^M \bPsi({\bx}_m) \bPsi({\bx}_m)^\top,
{\bf A}=\frac{1}{M}\sum_{m=1}^M \bPsi({\bx}_m) \bPsi({\mathbf y}_m)^\top,
\end{eqnarray}}
with ${\bf K},{\bf G},{\bf A}\in\mathbb{R}^{N\times N}$, $\|\cdot\|_F$ stands for Frobenius norm. The above least square problem admits an analytical solution 

\begin{eqnarray}
{\bf K}_{EDMD}=\bf{G}^\dagger \bf{A}\label{edmd_formula}.
\end{eqnarray}
Convergence results for EDMD algorithms in the limit as the number of data points and basis functions go to infinity are provided in \cite{korda2018convergence,klus2020eigendecompositions}.
In this paper, we work with Gaussian Radial Basis Function (RBF) for the finite-dimensional approximation of the linear operators.
Under the assumption that the basis functions are positive, like the Gaussian RBF, the NSDMD algorithm propose following convex optimization problem for the approximation of the Koopman operator that preserves positivity and Markov property in Property~\ref{property}. 
\begin{eqnarray}\label{nsdmd}
&\min\limits_{\hat{\bf P}}\parallel \hat{\bf G}{\hat{\bf P}}-\hat{\bf A}\parallel_F\\\nonumber
\text{s.t.} \;\;
& [\hat{\bf P}]_{ij}\geq 0,\;\;\;\hat {\bf P}\mathds{1} = \mathds{1},
\end{eqnarray}
where,
\begin{eqnarray}\hat{\bf G}={\bf G}{\bf \Lambda}^{-1},\;\;\hat{\bf A}={\bf A}{\bf \Lambda}^{-1},\;\;\&\;\;{\bf \Lambda}=\int_\bX {\bPsi}{\bPsi}^\top d\bx\label{hatGA},
\end{eqnarray}
with $\bf G$ and $\bf A$ are as defined in \ref{edmd1} and $\mathds{1}$ is a vector of all ones. All the matrices in Eq. (\ref{hatGA}) are pre-computed from the data. In fact, since the basis functions are assumed to be Gaussian RBF, the constant ${\bf \Lambda}$ matrix can be computed explicitly as 
\[\Lambda_{i,j} =(\frac{\pi\sigma^2}{2})^{n/2} \exp^\frac{-\lVert \bc_i-\bc_j\rVert^2}{2\sigma^2}, i,j=1,2,\ldots,N,\]
where $\bc_i,\bc_j$ are the centers of the $\psi_i$ and $\psi_j$ Gaussian RBFs respectively. 
The constraints in (\ref{nsdmd}) ensure that finite-dimensional approximation preserves the positivity property and Markov property respectively. The approximation for the P-F operator and its generator are obtained as the solution of the optimization problem (\ref{nsdmd}) as 
\begin{eqnarray}
\mathbb{P}_{\Delta t}\approx \hat{\bf P}^\top=: {\bf P},\;\;\;\;\;{\cal P}_{\bF}\approx \frac{\hat{\bf P}^\top-{\bf I}}{\Delta t}=:\bM.\label{PF_approximation}
\end{eqnarray}

\section{Main Results}\label{section_main}
We consider optimal control problem for  control affine system of the form
\begin{eqnarray}
\dot \bx=\bar {\bf f}(\bx)+{\bf g}(\bx)\bar u,\;\;\bx\in \bX\subseteq \mR^n,\label{cont_syst1}
\end{eqnarray}
where, the vector field $\bar {\bf f}, \bg \in {\cal C}^1(\bX,\mR^n)$ and $\bar u\in \mathbb{R}$ is the control input. For the simplicity of presentation we present results for the case of single input, the results generalize to the multi-input case in straight-forward manner. 
We make following assumption on the control system (\ref{cont_syst1}). 
\begin{assumption}\label{assume_stabilizability}
We assume that the linearization of the nonlinear control system at the origin i.e., the pair $(\frac{\partial \bar {\bf f}}{\partial \bx}(0), \bg(0))$, is stabilizable. 
\end{assumption}
Using the stabilizability Assumption \ref{assume_stabilizability} of the linearized dynamics, we design a local stabilizing controller. The procedure involves identifying the linearized dynamics using time-series data around the origin. We outline the details of this procedure in the computational Section \ref{compute_local}. 
Let $u_\ell$ be the local stabilizing controller.  Defining ${\bf f}:=\bar {\bf f}+\bg u_\ell$ and $u:=\bar u-u_\ell$, we can write system equation (\ref{cont_syst1}) as 
\begin{eqnarray}
\dot \bx= {\bf f}(\bx)+{\bf g}(\bx) u.\label{cont_syst}
\end{eqnarray}
Following is true for the control system (\ref{cont_syst}). 
The origin of system (\ref{cont_syst}) is locally stable with $u=0$. We denote the local domain of attraction around the origin by $\cal N$. Let $B_\delta\subset {\cal N}$ is a $\delta$ neighborhood of equilibrium point for any fixed $\delta >0$. We denote $\bX_1:=\bX\setminus B_\delta$. The objective is to design optimal controller which is active outside $B_\delta$. In Section \ref{section_localoptimal}, we outline procedure for combining the local and global optimal controller using a procedure of smooth blending controller discussed in \cite{rantzer2001smooth}. 
In the following, we assume that the measure $\mu_0\in {\cal M}(\bX)$ and is equivalent to Lebesgue with Radon–Nikod\'{y}m derivative $h$ i.e.,  $\frac{d\mu_0}{d\bx}=h_0(\bx)>0$ and $h_0\in {\cal L}_1(\bX,\mR_{> 0})\cap {\cal C}^1(\bX)$.

\subsection{Convex Formulation of Optimal Control}
Consider the optimal control problem with cost function of the form
\begin{eqnarray}
&J^\star(\mu_0)=\inf_{u}\int_{\bX_1}\int_0^\infty \left(q(\bx(t))+ r u^2(t)\right) \;dt d\mu_0(\bx)\nonumber\\
&{\rm subject\;to}\;\; (\ref{cont_syst}).\label{cost_function}
\end{eqnarray}

\noindent Some comments on the nature of the assumed cost function are necessary. Let 
\begin{eqnarray}V(\bx)=\int_0^\infty \left(q(\bx(t))+ru(t)^2 \right)dt,\label{cost_regular}
\end{eqnarray}
which is the cost function for the usual optimization problem formulated in the state space. 
Using (\ref{cost_regular}) the cost function (\ref{cost_function}) can be written as \[J(\mu_0)=\int_{\bX_1}V(\bx)d\mu_0(\bx)=\int_{\bX_1}V(\bx)h_0(\bx)d\bx,\]
where we used the fact that $d\mu_0(\bx)=h_0(\bx)d\bx$. Note that the new cost function in the density space is weighted with respect to the given density function $h_0$ associating different weightage to different initial conditions or sets.

Another distinguishing feature of the cost function is that the cost function is evaluated and minimized only on set $\bX_1$. Using Assumption \ref{assume_stabilizability}, we can design a local stabilizing controller and restrict the construction of optimal control outside the set $B_\delta$ around the origin. In Section \ref{section_localoptimal}, we outline a procedure for blending the local stabilizing controller with the global optimal controller smoothly. 
We make following assumptions on state cost function $q(\bx)$ and optimal control.
\begin{assumption}\label{assume_costfunction}
 We assume that the state cost function $q: \mR^n\to \mR_{\geq 0}$ is zero at the origin and uniformly bounded away from zero outside the neighborhood $B_\delta$ and $r>0$. 
\end{assumption}


\begin{assumption}\label{assume_OCP}
We assume that for the OCP (\ref{cost_function}) there exists a feedback control input, $u=k(\bx)$, such that the cost function corresponding to this input is finite. Furthermore, the optimal control is feedback in nature i.e., $u^\star=k^\star(\bx)$ with the function $k$ is assumed to be ${\cal C}^1(\bX)$. 
\end{assumption}

With the assumed feedback form of the control input, the OCP can be written as 

\begin{eqnarray}
J^\star(\mu_0)= \inf\limits_{k\in{\cal C}^1(\bX)} &\int_{\bX_1}\left[\int_0^\infty q(\bx(t))+ r k^2(\bx(t))\;dt\right] d\mu_0(\bx)\nonumber\\
 {\rm s.t.}&\dot \bx={\bf f}(\bx)+\bg(\bx)k(\bx).\label{ocp_main}
\end{eqnarray}
We now state the main theorem on the convex formulation of the OCP. 
\begin{theorem}\label{theorem_ocp}
Consider the optimal control problem (\ref{ocp_main}), with the  cost function and optimal control satisfying Assumptions  \ref{assume_costfunction}, and \ref{assume_OCP} respectively. The OCP (\ref{ocp_main}) can be written as following infinite dimensional convex  problem 

\begin{eqnarray}
J^\star(\mu_0)&=&\inf_{\rho\in {\cal S},\bar \rho\in {\cal C}^1(\bX_1)} \;\;\; \int_{\bX_1} q(\bx)\rho(\bx)+r \frac{\bar \rho(\bx)^2 }{\rho} d\bx\nonumber\\
{\rm s.t}.&&\nabla\cdot ({\bf f}\rho +\bg\bar \rho)=h_0, \label{eqn_ocp1}
\end{eqnarray}
where ${\cal S}:={\cal L}_1(\bX_1)\cap {\cal C}^1(\bX_1,\mR_{\geq 0})$. The optimal feedback control input is recovered from the solution of the above optimization problem as  
\begin{eqnarray}
k^\star(\bx)=\frac{\bar \rho^\star(\bx)}{\rho^\star(\bx)}\label{optimal_control},
\end{eqnarray}
where $(\rho^\star, \bar \rho^\star)$ are solution of (\ref{eqn_ocp1}). Furthermore, the optimal control $k^\star(\bx)$ is a.e. uniformly stabilizing the origin. 
\end{theorem}
\begin{proof}
Consider the feedback control system
\begin{eqnarray}\dot \bx={\bf f}(\bx)+\bg(\bx)k(\bx),\label{feedback_system}
\end{eqnarray}
where $u=k(\bx)$ be the feedback controller satisfying Assumption \ref{assume_OCP}. Let $\mP^c_t$ and $\mU^c_t$ be the P-F and Koopman operator for the feedback control system (\ref{feedback_system}). Using the definition of the Koopman operator, the cost in (\ref{ocp_main}) can be written as 
\[J(\mu_0)=\int_{\bX_1}\int_0^\infty [\mU_t^c (q+rk^2 )](\bx)dt h_0(\bx)d\bx.\]
Using the duality and linearity of the Koopman and P-F operators, we obtain
\begin{eqnarray}
J(\mu_0)=\int_{\bX_1}\int_0^\infty (q+rk^2)(\bx)[\mP_t^c h_0](\bx)dt d\bx.\label{ocp_costproof}
\end{eqnarray}
Using the Assumption \ref{assume_OCP} on the existence of optimal control for which the optimal cost is finite, we obtain
\begin{eqnarray}
\kappa \int_{\bX_1}\int_0^\infty [\mP_t h](\bx)dt d\bx \leq J(\mu_0)\leq M <\infty, \label{rr}
\end{eqnarray}
for some constant $M>0$, where $\kappa$ is assumed lower bound on the  state cost  $q(\bx)$ following Assumption \ref{assume_costfunction}. We define
\begin{eqnarray}\rho(\bx):=\int_0^\infty [\mP_t^c h_0](\bx) dt. \label{definingrho1}
\end{eqnarray} 
It follows from (\ref{rr}) that $\rho(\bx)$ is well defined for a.e. $\bx$ and is an integrable function. From the definition of the P-F operator and the assumption made on function $h_0$  it follows that the function $[\mP_t^c h_0](\bx)$ is uniformly continuous function of time. Hence using Barbalat Lemma \cite[pg. 269]{barbalat1959systemes} we have 
\begin{eqnarray}
\lim_{t\to \infty}[\mP_t h_0](\bx)=0,\label{convergence1}
\end{eqnarray}
for a.e. $\bx$.




Now using the definition of $\rho$ from (\ref{definingrho1}) we  write (\ref{ocp_costproof}) as
\begin{eqnarray}J(\mu_0)=\int_{\bX_1}(q+rk^2)(\bx)\left[\int_0^\infty [\mP_t h_0](\bx)dt\right] d\bx\nonumber\\
=\int_{\bX_1}(q+rk^2)(\bx)\rho(\bx) d\bx.\label{costcc}
\end{eqnarray}
Defining $\bar\rho(\bx):=\rho(\bx)k(\bx)$ it follows that  (\ref{costcc}) can be written in the form  (\ref{eqn_ocp1}).  
We next show that $\rho(\bx)$ and $\bar \rho$ satisfies the constraints in (\ref{eqn_ocp1}). 
Substituting (\ref{definingrho1}) in the constraint of (\ref{eqn_ocp1}), we obtain
\begin{eqnarray}
&&\nabla\cdot ({\bf f}_c(\bx) \rho(\bx))=\int_0^\infty \nabla\cdot ({\bf f}_c(\bx)[\mathbb{P}_t^c h_0](\bx)) dt\nonumber\\
&=&\int_0^\infty -\frac{d}{dt}[\mathbb{P}^c_t h_0](\bx)dt=-[\mathbb{P}^c_th_0](\bx)\Big|^{\infty}_{t=0}=h_0(\bx),\nonumber\\\label{eq11}
\end{eqnarray}
where, ${\bf f}_c(\bx):={\bf f}(\bx)+\bg(\bx)k(\bx)$. In deriving (\ref{eq11}) we have used the infinitesimal generator property of P-F operator Eq. (\ref{PF_generator}) and the fact that $\lim_{t\to \infty} [\mathbb{P}_t^c h_0](\bx)=0$  following (\ref{convergence1}).
Furthermore, since $h_0> 0$, it follows that $\rho> 0$ from the positivity property of the P-F operator. Combining (\ref{ocp_costproof}) and (\ref{eq11}) along with the definition of $\bar \rho(\bx):=\rho(\bx) k(\bx)$, it follows that the OCP problem can be written as  convex optimization problem (\ref{eqn_ocp1}). The optimal control $k^\star(\bx)$ from (\ref{optimal_control}) is a.e. uniformly stabilizing w.r.t. measure $\mu_0$ follows from the results of Theorem \ref{theorem_necc_suff} and using the fact that closed loop system satisfies (\ref{eq11}). $\bar \rho(\bx)\in {\cal L}_1(\bX_1)$. $\rho(\bx)\in {\cal L}_1(\bX_1)\cap {\cal C}^1(\bX_1,\mR_{\geq 0})$ follows from the fact that $h_0\in {\cal C}^1(\bX)$, the integral formula for $\rho$ in (\ref{definingrho1}), and  the definition of P-F operator (\ref{PF_operator}).
\end{proof}

\begin{remark}
The optimal control problem's goal is to minimize the cost function; however, the optimal control is also stabilizing the closed-loop dynamics in a.e. sense. The optimal density function serves as a stability certificate for the feedback control system. 
\end{remark}

We next consider the optimization problem involving $1$-norm on the control input. 
{\small
\begin{eqnarray}
 \inf\limits_{k\in{\cal C}^1(\bX)} &\int_{\bX_1}\left[\int_0^\infty q(\bx(t))+ r|k(\bx(t))|\;dt\right] d\mu_0(\bx)\nonumber\\
 {\rm s.t.}&\dot \bx={\bf f}(\bx)+\bg(\bx)k(\bx).\label{ocp_mainl1norm}
\end{eqnarray}}
We make following assumption
\begin{assumption}\label{assume_OCPl1}
We assume that for the OCP (\ref{ocp_mainl1norm}) there exists a feedback control input such that the  cost function corresponding to the input is finite. Furthermore, the optimal control is feedback in nature i.e., $u^\star=k^\star(\bx)$ with the function $k$ is assumed to be ${\cal C}^1(\bX)$. 
\end{assumption}
 
\begin{theorem}\label{theorem_ocpl1}
Consider the optimal control problem (\ref{ocp_mainl1norm}), with the cost function, and optimal control satisfying Assumptions \ref{assume_costfunction} and \ref{assume_OCPl1} respectively. The OCP (\ref{ocp_main}) is written as following infinite dimensional optimization problem 
{\small
\begin{eqnarray}
J^\star(\mu_0)&=&\inf_{\rho\in {\cal S},\bar \rho\in {\cal C}^1(\bX_1)} \;\;\; \int_{\bX_1} q(\bx)\rho(\bx)+r |\bar\rho(\bx)| d\bx\nonumber\\
{\rm s.t}.&&\nabla\cdot ({\bf f}\rho +\bg\bar \rho)=h_0, \label{eqn_ocp12}
\end{eqnarray}}
where ${\cal S}:={\cal L}_1(\bX_1)\cap {\cal C}^1(\bX_1,\mR_{\geq 0})$. The optimal feedback control input is recovered from the solution of the above optimization problem as  
$k^\star(\bx)=\frac{\bar \rho^\star(\bx)}{\rho^\star(\bx)}$.
Furthermore, the optimal control $k^\star(\bx)$ is a.e. uniformly stabilizing the origin. 
\end{theorem}
\begin{proof}
The proof of this theorem follows along the lines of proof of Theorem \ref{theorem_ocp}.\hfill{$\blacksquare$} 
\end{proof}



\subsection{Local Optimal Control and Controller Blending}\label{section_localoptimal}
The density function $\rho$ for the solution of optimization problem satisfy 
\[\rho(\bx)=\int_0^\infty [\mathbb{P}^c_t h_0](\bx) dt\]
where $\mathbb{P}_t^c$ is the P-F operator for the closed-loop system $\dot \bx={\bf f}(\bx)+{\bf g}(\bx){ k}(\bx)$ and hence $\rho$ serves as an occupancy measure i.e., $\int_A \rho(\bx) d\bx=\left<\int_0^\infty[\mathbb{U}_t \chi_A] dt,h\right>$ for any set $A\in {\cal B}(\bX_1)$ signifies the amount of time closed-loop system trajectories spend in the set $A$ with initial condition supported on measure $\mu_0$, where $\chi_A$ is the indicator function of set $A$. Because of this, $\rho(\bx)$ has a singularity at the equilibrium point stabilized by the closed-loop system.  We exclude the small neighborhood around the origin for the proper parameterization of the density function $\rho$ in the computation of optimal control due to singularity at the origin. In particular, the optimization problem (\ref{eqn_ocp1}) is solved excluding the small neighborhood around the origin. For the small region around the origin, we either design stabilizing controller or optimal control depending upon the linearized system's stabilizability or controllability property around the origin. For the design of the local controller we use time-series for the identification of linear dynamics. 
Let $k(\bx)$ and $\rho(\bx)$ be the global optimal controller obtained as the solution of optimization problem (\ref{eqn_ocp1}). Let $k_{\ell}$ be the local feedback controller obtained by solving linear quadratic regular (LQR) optimal control or Lyapunov  based stabilization control  using the identified linearized dynamics and $\rho_{\ell}(\bx)=\max\{(\bx^\top P\bx)^{-3}-\gamma,0\}$, where $P$ is a solution of Riccatti equation or controlled Lyapunov equation obtained based on identified linearized dynamics. The local and global controllers are combined using the blending procedure from \cite{rantzer2001smooth} as follows:
\[\bar u=\frac{\rho_{\ell}(\bx)}{\rho_{\ell}(\bx)+\rho(\bx)} k_{\ell}(\bx)+\frac{\rho(\bx)}{\rho_{\ell}(\bx)+\rho(\bx)} k(\bx),\]
 The region where local controller is active is given by $\bx^\top P\bx\leq (\frac{1}{\delta})^{\frac{1}{3}}$. 


\subsection{Data-Driven Nonlinear Stabilization}

The data-driven stabilization of a nonlinear system will be the particular case of our data-driven optimal control system. As the constraint in the optimization problem guarantees stability following the results of Theorem \ref{theorem_necc_suff}. In particular, using the optimal control formulation in (\ref{eqn_ocp1}), the stabilization problem can be posed as a  feasibility problem. Our proposed data-driven approach for stabilization will stand in contrast to the model-based approach for stabilization using density function as presented in \cite{Prajna04}.


\section{Data Driven Approximation}\label{section_data-driven}

For the data-driven computation of  optimal control, we need to provide finite dimensional approximation of the infinite dimensional linear program (\ref{eqn_ocp1}) and (\ref{eqn_ocp12}). Towards this goal we need the data-driven approximation of the generator corresponding to vector field $\bf f$ and $\bg$ i.e., $\nabla\cdot ({\bf f}\rho)$ and $\nabla\cdot(\bg\bar \rho)$. 
\begin{remark}\label{remark_positivebasis}
In this paper, we use Gaussian RBF to obtain all the simulation results i.e.,
$
\psi_k(\bx)=\exp^{-\frac{\parallel \bx-{\bf c}_k\parallel^2}{2\sigma^2}}.$
where $\bc_k$ is the center of the $k^{th}$ Gaussian RBF. 
\end{remark}
\subsection{Approximation of Convex Optimization Problem}
Let ${\bf P}_0\in \mathbb{R}^{N\times N}$  be the finite-dimensional approximation of the P-F operator corresponding to uncontrolled dynamical system $\dot \bx={\bf f}(\bx)$. Similarly, let ${\bf P}_1$ be the P-F operator of the system obtained with unit step input i.e., $\dot \bx={\bf f}(\bx)+\bg(\bx)$. These operators are obtained using NSDMD algorithm from section \ref{section_nsdmd} with time series data generated from the dynamical system with discretization time-step of $\Delta t$.
The approximation of the P-F generator corresponding to the vector field ${\bf f}$ is  
\begin{eqnarray}
{\cal P}_{\bf f}\approx \frac{1}{\Delta t}(\bP_0-{\bf I})=:\bM_0. \label{P0_approx}
\end{eqnarray}
Using linearity property of the generator it follows that 
\begin{eqnarray}
{\cal P}_{\bg}={\cal P}_{{\bf f}+\bg}-{\cal P}_{\bf f}\approx \frac{\bP_1-\bP_0}{\Delta t}=:\bM_1. \label{Pj_approx}
\end{eqnarray}
Let $h(\bx)$, $ \rho(\bx)$, and $\bar \rho(\bx)$ be expressed in terms of the basis function 
\begin{eqnarray}
h(\bx)=\bPsi^\top {\bf m},\;\rho(\bx)\approx \bPsi^\top{\bf v},\;\bar\rho(\bx)\approx \bPsi^\top {\bf w}.\label{approx}
\end{eqnarray}
With the above approximation of the generators ${\cal P}_{\bf f}$ and ${\cal P}_{\bg_i}$ and $\rho, \bar \rho$ we can approximate the equality constraints in the optimization problem (\ref{eqn_ocp1})  as finite dimensional equality constraints. 
\[
-\bPsi(\bx)^\top \left(\bM_0 {\bf v}+ \bM_1{\bf w}\right)=\bPsi(\bx)^\top {\bf m}.\]
We now proceed with the approximation of the cost function. 
\[\int_{\bX_1}q(\bx)\rho(\bx)d\bx\approx \int_{\bX_1} q(\bx)\bPsi^\top d\bx {\bf v}={\bf d}^\top{\bf v},\]
where the vector ${\bf d}:=\int_X q(\bx)\bPsi d\bx$ can be pre-computed.  
We use following approximation for the term 
\[\frac{\bar \rho(\bx)}{\rho(\bx)}\approx \bPsi(\bx)^\top \frac{\bw}{\bv}, \]
where the division is assumed element-wise. The above approximation is justified as we use Gaussian RBFs as basis functions in our simulation and apply following thump rule to select parameters for Gaussian RBFs, $d\leq 3\sigma\leq 1.5 d$. Where $d$ is the distance between the Gaussian RBFs centers and $\sigma$ is the standard deviation. The choice of parameters ensure that in the region of intersection of two Gaussian RBFs, the function takes smaller values. 
With the above approximation, we have
\[\frac{\bar \rho^2}{\rho}\approx \bw^\top \bPsi \bPsi^\top\frac{\bw}{\bv},\;\;\int_{\bX_1}\frac{\bar \rho^2}{\rho}d\bx=\bw^\top {\bf D}\frac{\bw}{\bv},\]
where, ${\bf D}=\int_{\bX_1} \bPsi\bPsi^\top d\bx $. We have the following approximation to the optimization problem (\ref{eqn_ocp1})

\begin{eqnarray*}
&\min_{\bPsi^\top{\bf v}\geq 0, {\bf w}} {\bf d}^\top {\bf v}+r {\bf w}^\top {\bf D} \frac{{\bf w}}{{\bf v}}\\
&{\rm s.t.}\;-\bPsi(\bx)^\top \left(\bM_0 {\bf v}+ \bM_1{\bf w}\right)=\bPsi(\bx)^\top {\bf m}.
\end{eqnarray*}

Since the basis functions are taken to be positive, (Remark \ref{remark_positivebasis}), the approximation for the $\rho$ and $\bar \rho$ in (\ref{approx}) can be obtained by solving following finite-dimensional convex problem.

\begin{eqnarray}
&\min_{{\bf v}\geq 0, {\bf w}} {\bf d}^\top {\bf v}+ r {\bf w}^\top {\bf D} \frac{{\bf w}}{{\bf v}}\nonumber\\
&{\rm s.t.}\; -\left(\bM_0 {\bf v}+ \bM_1{\bf w}\right)= {\bf m}.\label{cvxopt_1}
\end{eqnarray}
 The optimization problem is convex as the cost function is quadratic over linear  and the constraints are linear in the decision variables.
The optimal control is then approximated as $u=\bPsi^\top(\bx) \frac{{\bf w}}{{\bf v}}$, where the division is element-wise. 
Similarly, the finite dimensional approximation of the OCP in Theorem \ref{theorem_ocpl1} corresponding to ${\cal L}_1$ norm on control is given by 
\begin{eqnarray}
&\min_{{\bf v}\geq 0, {\bf w}} {\bf d}^\top {\bf v}+rc |{\bf w}|\nonumber\\
&{\rm s.t.}\; - \left(\bM_0 {\bf v}+ \bM_1{\bf w}\right)={\bf m},\label{cvxopt_2}
\end{eqnarray}
where $c=\int_{\bX_1}\psi_i(\bx)d\bx=\int_{\bX_1}\psi_j(\bx)d\bx$ is a positive constant. 

\subsection{Computation of Local Optimal Controller}\label{compute_local}
For the computation of the local optimal controller, we identify local linearized dynamics from data. To identify the linearized dynamics, we use time-series data generated by initializing the system around the origin. For the approximation of local linear dynamics, we use the EDMD algorithm instead of NSDMD with basis function $\bPsi(\bx)=\bx$, i.e., identity function. 
In particular, let $\bA$ and $\bf b$ are the matrices for the local linear approximation of the system dynamics in discrete-time is obtained by solving following least square problem
\begin{eqnarray}
\min_{\bA,{\bf b}}\parallel {\cal Y}-\bA {\cal X} -{\bf b}{\cal U}\parallel,
\end{eqnarray}
where,{\small
\[{\cal Y}:=[{\bf y}_1,\ldots,{\bf y}_{L},{\bf y}_{L+1},\ldots{\bf y}_M],{\cal X}=[{\bf x}_1,\ldots,{\bf x}_{L},{\bf x}_{L+1},\ldots{\bf x}_M],\]} ${\cal U}=[\underbrace{0,\ldots, 0}_{L},\underbrace{1,\ldots,1}_{M-L}]$ and ${\bf y}_k=\bs_{\Delta t}(\bx_k,\bar u\in \{0,1\})$ is the solution of (\ref{cont_syst1}) with either zero input or step input and initial condition $\{\bx_k\}$ initialized around the origin. So part of the data is generated using zero input and remaining using step input.  
Once we have the local approximation of the system matrices, the local controller, $u_\ell$ is designed using the Lyapunov-based approach for stabilization or using the linear quadratic regulator (LQR) based optimal control. 

\section{Simulation results}\label{secction_simulation}
 All the simulation results in this paper are obtained using Gaussian RBF. The following rules of thumb are abided in selecting centers and $\sigma$ parameters for the Gaussian RBF. The RBF centers are chosen to be uniformly distributed in the state space at a distance of $d$. The $\sigma$ for the Gaussian RBF is chosen such that $d\leq 3\sigma\leq 1.5 d$. All the simulation results are performed using MATLAB on a desktop computer with 16GB RAM total simulation time for each of these examples did not exceed more than five minutes. The optimization problem is solved using CVX. \\
 
\subsection{ Scalar Nonlinear System} 
The first example of scalar system is chosen to compare the optimal control obtained using our proposed data-driven approach with the analytical derived optimal control.
 \[
 \dot x = a x^3+u,\;\;a=0.5.
 \]
 The analytical formula for optimal control obtained by solving HJB equation with cost function  $x^2+u^2$ is $u^\star=-ax^{3}- x\sqrt{a^2 x^4+1}$. In Fig. \ref{fig:lineara} shows the comparison of the closed loop trajectories obtained using feedback controller $u=u^\star$ and controller obtained using data-driven computational framework. 
For the finite dimensional approximation we use $5$ Gaussian RBF with $\sigma = 1.225$, and the centers of basis functions are distributed uniformly within the range of $D= [-5,\; 5]$. The region of operation for the blending control is $|x|\leq 0.15$ and is marked with purple line in Fig. \ref{fig:lineara}. For the approximation of P-F operator, we applied NSDMD algorithm using one-step time-series data with $1\text{e}3$  initial conditions, with $\Delta t= 0.01$. The comparison between the analytical derived optimal control and P-F based optimal control for the state trajectories and optimal cost shows a close match. 
 \begin{figure}[htbp]
\centering
\begin{subfigure}[b]{0.35\textwidth}
\includegraphics[width=\textwidth]{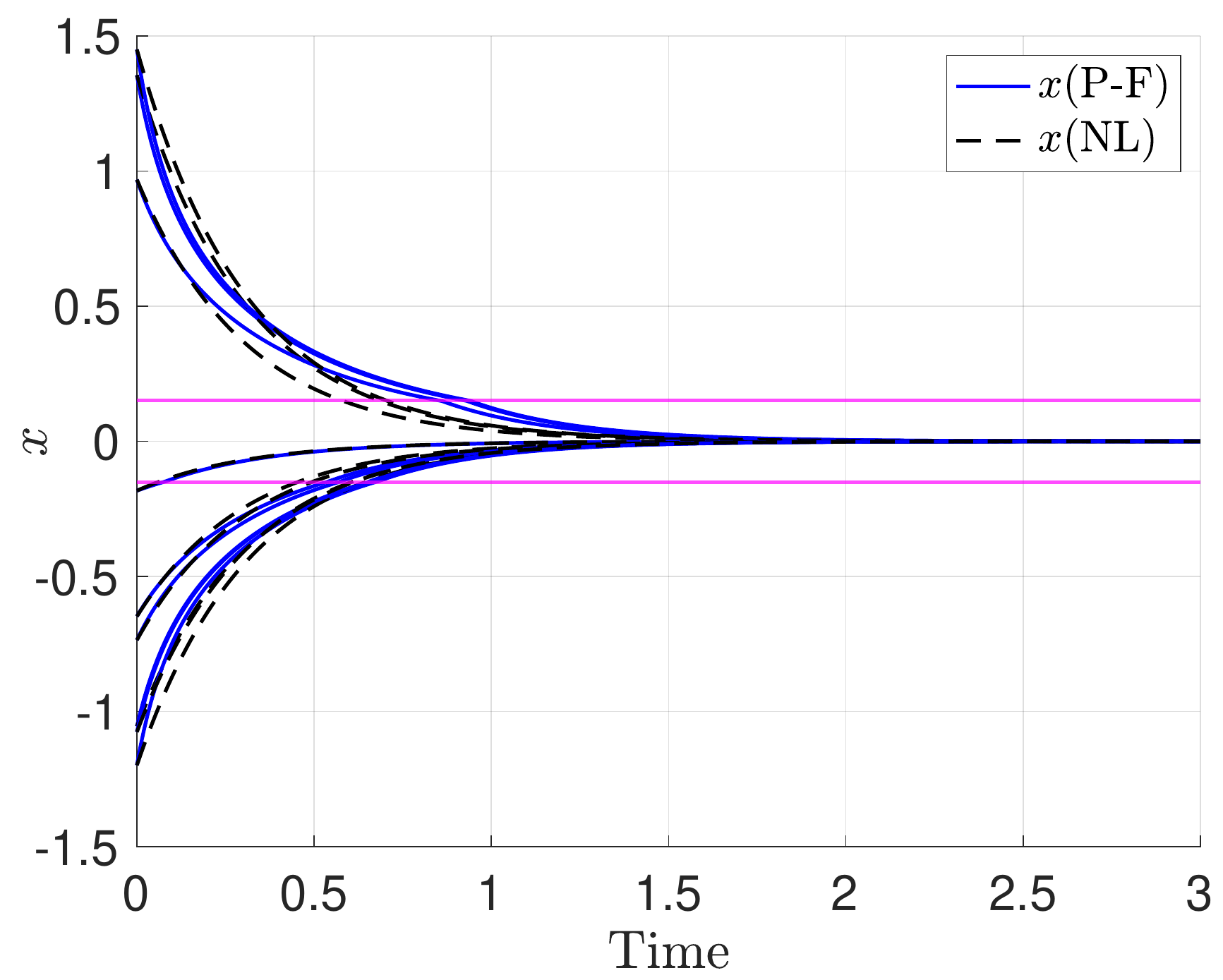}
\caption{$x_{1}$ vs $t$}
\label{fig:lineara}
\end{subfigure}
\hfill
\begin{subfigure}[b]{0.35\textwidth}
\includegraphics[width=\textwidth]{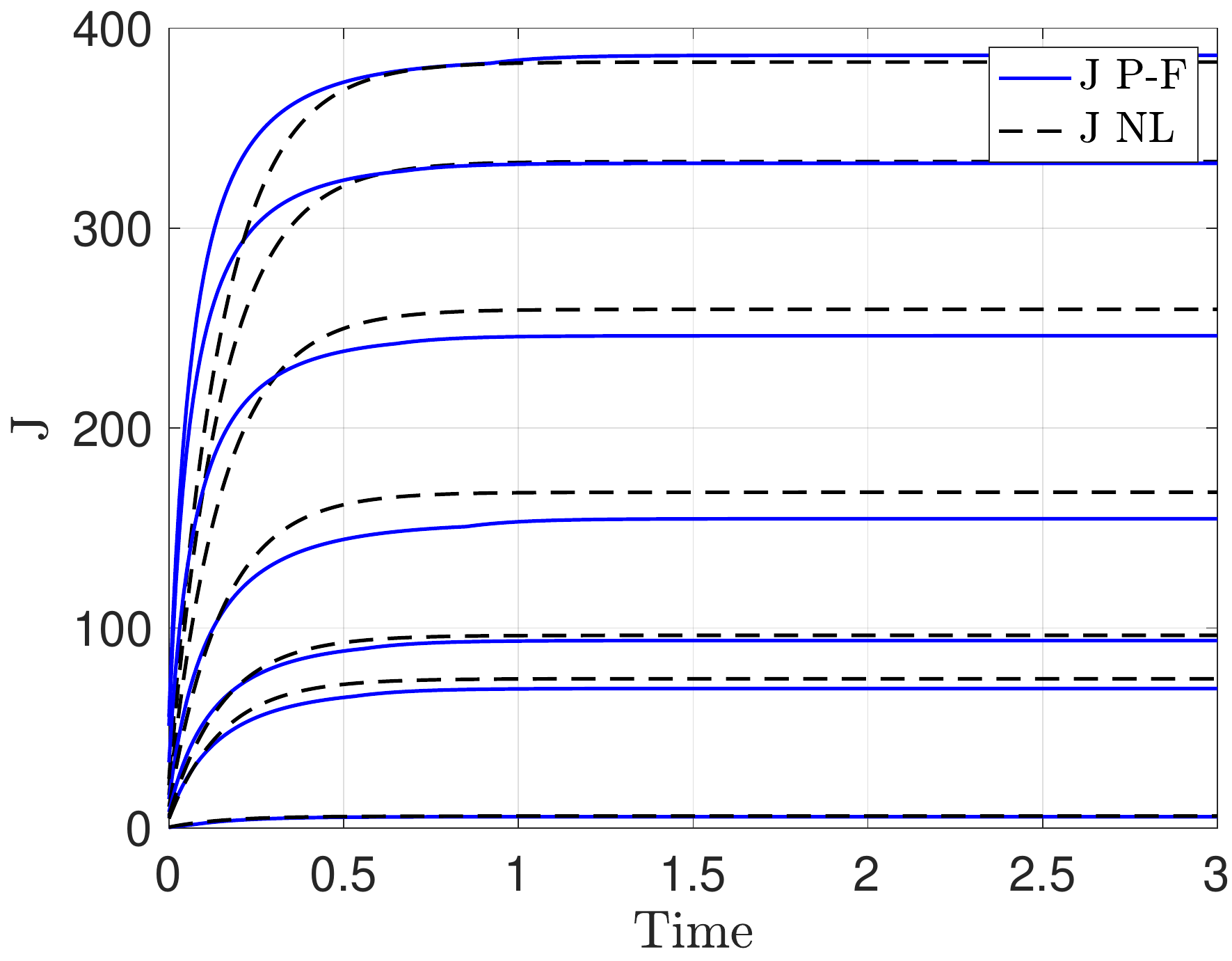}
\caption{Comparison of cost function}
\label{fig:linearb}
\end{subfigure}
\caption{Comparison of state trajectories and cost obtained using analytical formula and P-F based data-driven control.}
\end{figure}



\subsection{  Duffing Oscillator}
\begin{equation}
\begin{aligned}
    \dot{x}_1 &= x_2,\;\;
    \dot{x}_2 &= -x_1 - x_1^3 - 0.5x_2 +u. \\
\end{aligned}
\end{equation}
In this example, we used 225 Gaussian RBF with $\sigma = 0.21$, and the centers of basis functions are distributed uniformly within the range of $D= [-3,\; 3]\times [-3,\;3]$, $r=1000$. The cost function is chosen to be $\bx^\top \bx+u^2$. 
 For the approximation of P-F operator, we applied NSDMD algorithm using one-step time-series data with $5\text{e}4$ initial conditions, $\Delta t= 0.01$. The region where the blending controller is active is marked by doted ellipsoid around the origin in Fig. \ref{fig:vdp2_1}. Simulation results show that the optimal control is successful in stabilizing the origin. \\
 \begin{figure}[htbp]
\centering
\begin{subfigure}[b]{0.35\textwidth}
\includegraphics[width=\textwidth]{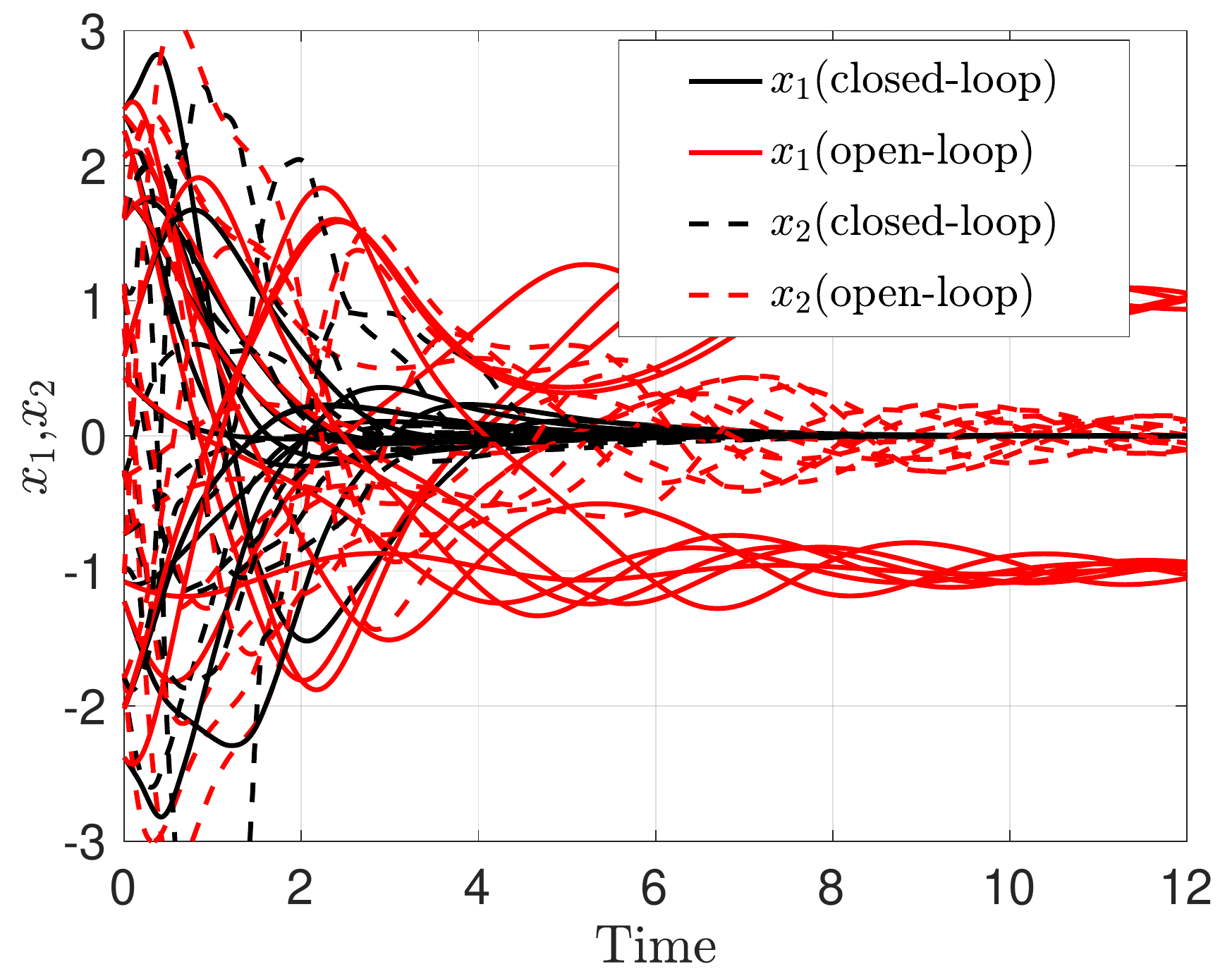}
\caption{$x_{1\sim 2}$ vs $t$}
\label{fig:vdp2_1}
\end{subfigure}
\begin{subfigure}[b]{0.35\textwidth}
\includegraphics[width=\textwidth]{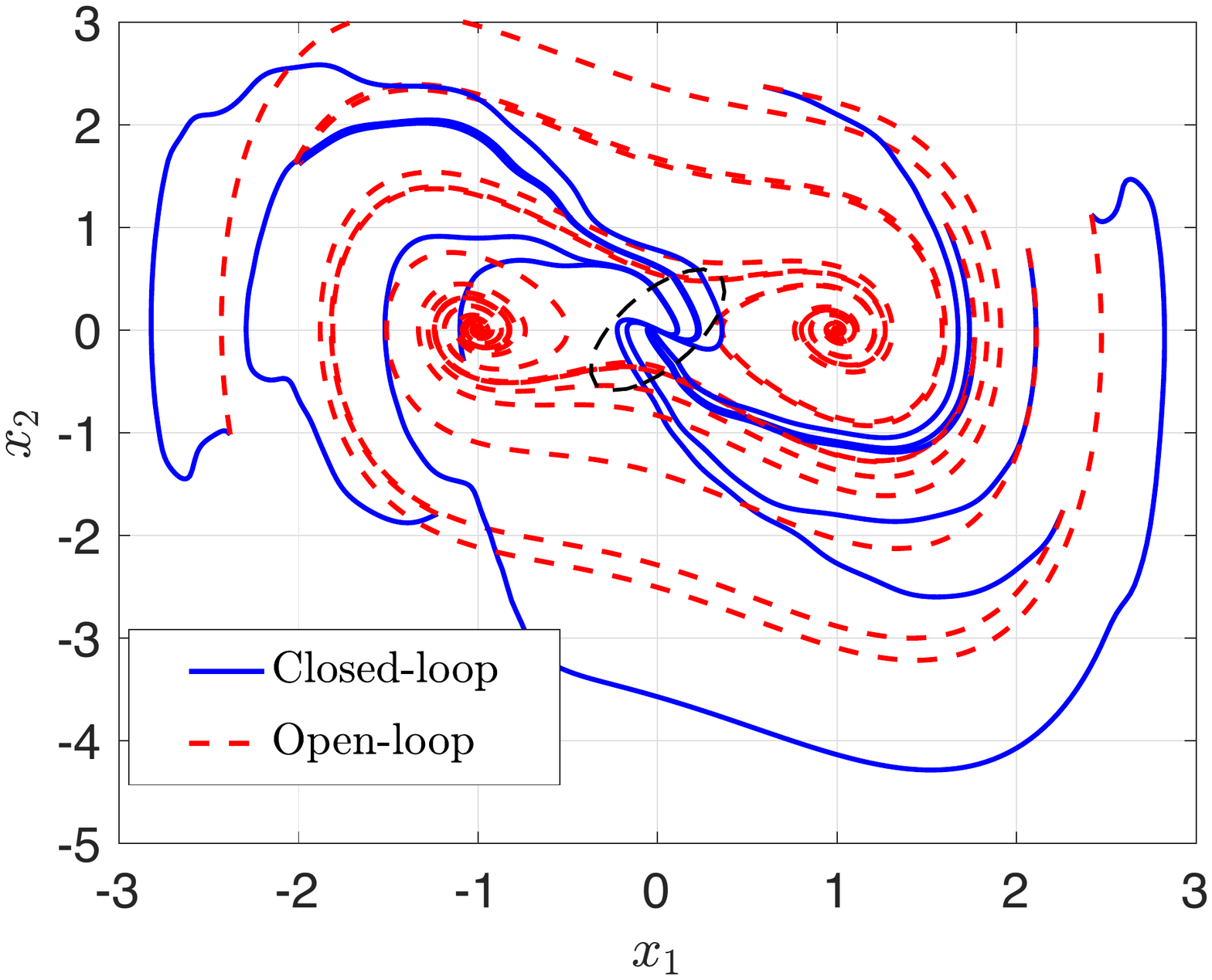}
\caption{Trajectories in 2-D space}
\label{fig:vdp2_2}
\end{subfigure}
\caption{Optimal Control of Duffing Oscillator.}
\end{figure}
\subsection{  Controlled 3D Vanderpol Oscillator}

\begin{equation}
\begin{aligned}
    \dot{x}_1 &= x_2\\
    \dot{x}_2 &= -x_1 + x_2 - x_3 - x_1^2x_2 \\
    \dot{x}_3 &= x_3 -x_3^2 + 0.5 u.
\end{aligned}
\end{equation}
For  this example, we are using 512 Gaussian RBF with $\sigma = 0.14$, and the centers of basis functions are distributed uniformly within the range of $D= [-1,\; 1]\times [-1,\;1]\times [-1,1]$. The cost function is chosen to be $\bx^\top \bx+u^2$. 
 For the approximation of P-F operator, we applied NSDMD algorithm using one-step time-series data with $5e4$ initial conditions, $\Delta t= 0.01$. Fig. \ref{fig:vdp2_1} shows the results for the optimal control of 3D oscillator system with 10 initial conditions.  
 \begin{figure}[htbp]
\centering
\begin{subfigure}[b]{0.35\textwidth}
\includegraphics[width=\textwidth]{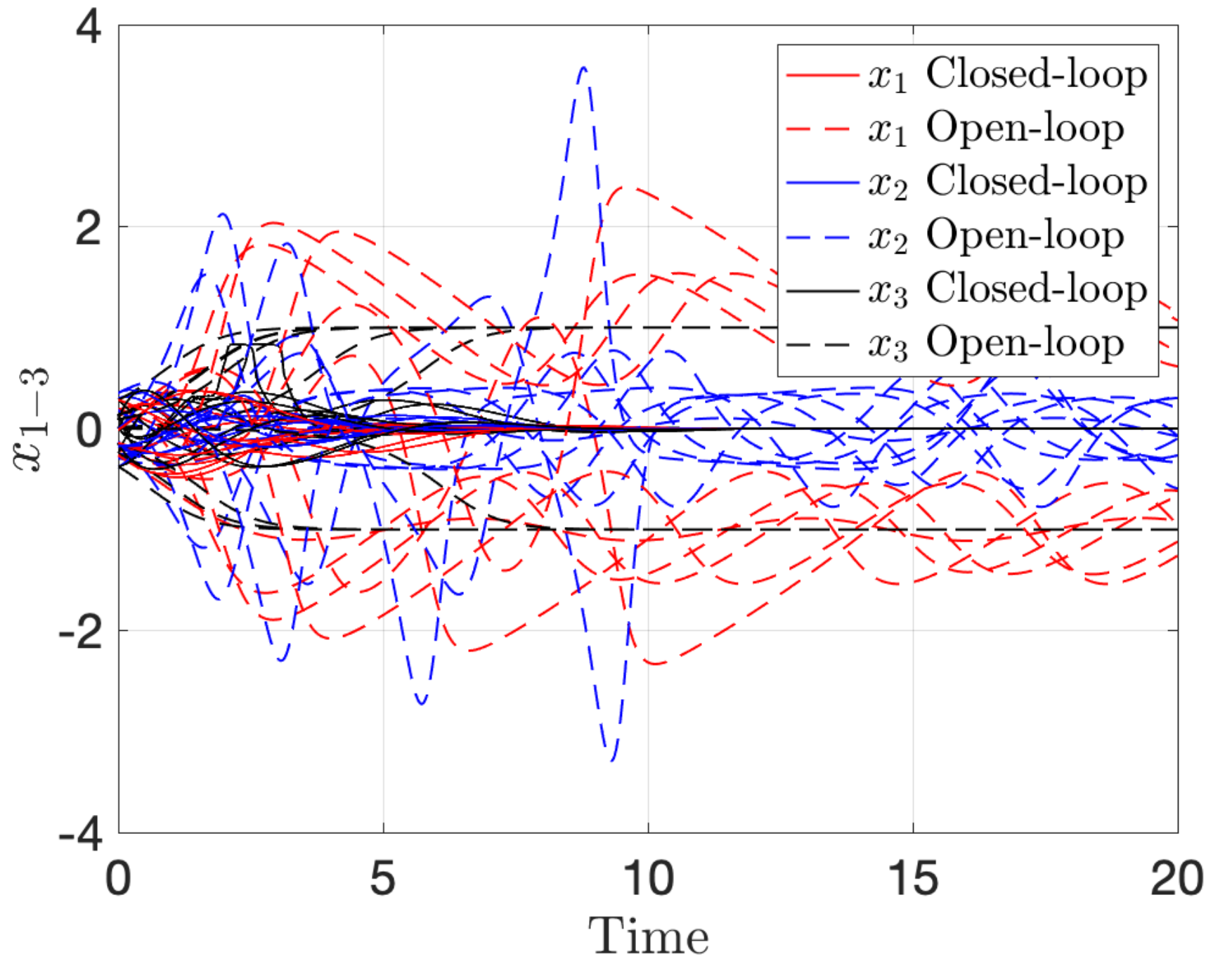}
\caption{$x_{1\sim 2}$ vs $t$}
\label{fig:vdp2_1}
\end{subfigure}
\begin{subfigure}[b]{0.35\textwidth}
\includegraphics[width=\textwidth]{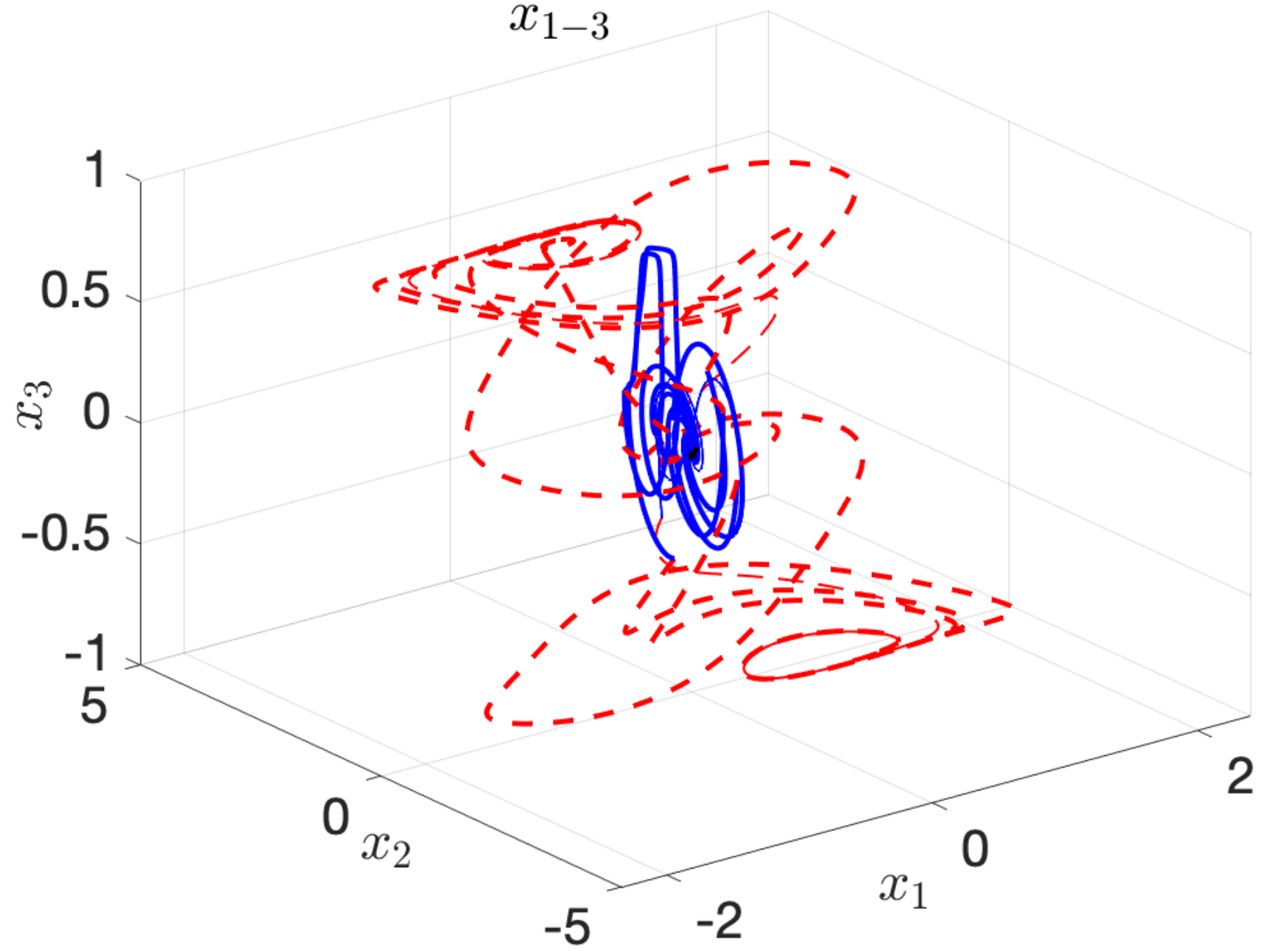}
\caption{Trajectories in 3-D space}
\label{fig:vdp2_2}
\end{subfigure}
\caption{Optimal Control of Oscillator}
\end{figure}

\section{Conclusion}\label{section_conclusion}
In this paper, we have provided a convex optimization-based formulation for the infinite horizon optimal control problem in the dual density space. We provided a data-driven approach for the computation of optimal control. The data-driven approach is based on the approximation of the P-F and Koopman operator for the finite-dimensional approximation of the convex optimization problem. Future research efforts will focus on the development of a computationally efficient numerical scheme and the choice of appropriate basis function for the implementation of the developed algorithm to system with large dimensional state space.

\bibliographystyle{IEEEtran}
\bibliography{ref1,ref}

\end{document}